\documentclass[10pt,a4paper]{article}
\usepackage{fullpage}
\usepackage{amsfonts,amsmath,amssymb}
\usepackage{amsthm}
\usepackage{graphicx}\usepackage{sectsty}

\theoremstyle{plain}
\newtheorem{theorem}{Theorem}[section]
\newtheorem{lemma}[theorem]{Lemma}
\newtheorem{proposition}[theorem]{Proposition}

\theoremstyle{definition}
\newtheorem{definition}[theorem]{Definition}

\theoremstyle{remark}

\numberwithin{equation}{section}

\sectionfont{\large}

\begin{document}
\title{An Uniqueness Theorem on the Eigenvalues\\of Spherically Symmetric Interior Transmission
Problem\\in Absorbing Medium}
\author{Lung-Hui Chen$^1$}\maketitle\footnotetext[1]{Department of
Mathematics, National Chung Cheng University, 168 University Rd.
Min-Hsiung, Chia-Yi County 621, Taiwan. Email:
mr.lunghuichen@gmail.com;\,lhchen@math.ccu.edu.tw. Fax:
886-5-2720497.}
\begin{abstract}
We study the asymptotic distribution of the eigenvalues of
interior transmission problem in absorbing medium. We apply
Cartwright's theory and the technique from entire function theory to find a Weyl's type of density theorem in absorbing medium. Given a sufficient quantity of
transmission eigenvalues, we obtain limited uniqueness on the
refraction index as an uniqueness problem in entire function
theory.
\end{abstract}
\section{Preliminaries}
In this note, we study the eigenvalues of the interior
transmission problem with a twice differentiable absorbing
refraction index
\begin{equation}
n_1(x):=\epsilon_1(x)+i\frac{\gamma_1(x)}{k}:
\end{equation}
\begin{eqnarray}\label{1.1}
\left\{%
\begin{array}{ll}
    \Delta w+k^2(\epsilon_1(x)+i\frac{\gamma_1(x)}{k})w=0,  & \hbox{ in }B; \\
    \Delta v+k^2(\epsilon_0+i\frac{\gamma_0}{k})v=0 & \hbox{ in }B; \\
    w=v, & \hbox{ on }\partial B; \\
    \frac{\partial w}{\partial r}=\frac{\partial v}{\partial r}& \hbox{ on }\partial B,\\
\end{array}%
\right.
\end{eqnarray}
where $r:=|x|$ and $B:=\{|x|\leq 1,\,x\in\mathbb{R}^3\}$, $w,v\in L^2(B)$,
$w-v\in H^2_0(B)$, $k\in\mathbb{C}$.  We consider the spherical
perturbations for~(\ref{1.1}) by setting
$\epsilon_1(x)=\epsilon_1(r)>0$ and $\gamma_1(x)=\gamma_1(r)>0$,
$\forall r\in[0,1]$; $\epsilon_0$ and $\gamma_0$ are positive
constants and $n_1(r)=\epsilon_0+i\frac{\gamma_0}{k}$, when $r\geq1$.
\par
The interior transmission eigenvalues play a role in the inverse
scattering theory both in numerical computation and in theoretical
purpose. See Colton and Monk \cite{Colton}, Colton and Kress \cite{Colton2} and Colton,
P\"{a}iv\"{a}rinta and Sylvester \cite{Colton3} for the historic
and theoretical context. Moreover, the eigenvalues of the interior
transmission problem is directly connected to the zeros of
scattering amplitude. They are zeros of the integral average of
the scattering amplitude. We refer to McLaughlin and Polyakov
\cite{Mc}. Moreover, It is another research interest to study the distribution of interior transmission eigenvalues in $\mathbb{C}$ \cite{Cakoni,Cakoni2,Colton,Colton2,Kirsch,Mc}. It is expected to prove a Weyl's type of asymptotics on these eigenvalues. In this paper we use the analysis in Levin \cite{Levin} to discuss the zeros of an asymptotically almost periodic function
along the real axis.

\par
Let us consider the solutions of~(\ref{1.1}) of the
following form:
\begin{eqnarray}\label{1.2}
&&v(r)=c_1j_0(k\tilde{n}_0r);\\
&&w(r)=c_2\frac{y(r)}{r}, \label{1.3}
\end{eqnarray}
where
$\tilde{n}_0:=(\epsilon_0+i\frac{\gamma_0}{k})^{\frac{1}{2}}$,
$j_0$ is a spherical Bessel function of order zero and $y(r)$ is a
solution of
\begin{eqnarray}\label{15}
\left\{
  \begin{array}{ll}
    y''+k^2(\epsilon_1(r)+i\frac{\gamma_1(r)}{k})y=0; \vspace{5pt}\\
    y(0)=0, \, y'(0)=1.
  \end{array}
\right.
\end{eqnarray}
The existence of $c_1$, $c_2$ in~(\ref{1.2}),~(\ref{1.3}) is provided by
\begin{equation}
D(k):=\det\left(%
\begin{array}{cc}
  y(1) & -j_0(k\tilde{n}_0) \\
  \{\frac{y(r)}{r}\}'|_{r=1} & -j_0'(k\tilde{n}_0r)|_{r=1} \\
\end{array}%
\right)=0. \label{1.7}
\end{equation}
The computation on  $c_1$, $c_2$ is discussed in \cite{Colton,Mc}.
Thus, the interior transmission eigenvalues are the zeros of such a
functional determinant. Furthermore, it is well-known that $D(k)$
is an entire function of exponential type bounded on the real
axis. The analysis of $D(k)$ is found in \cite[Sec.\,3]{Cakoni}. When we set $a=1$
 in the setting in \cite{Cakoni},
\begin{equation}\label{188}
D(k)=\frac{1}{ik[\epsilon_1(0)\epsilon_0]^{\frac{1}{4}}}\sinh[ik\sqrt{\epsilon_0}
-ik\int_0^1\sqrt{\epsilon_1(\rho)}d\rho
-\frac{\gamma_0}{2\sqrt{\epsilon_0}}+\frac{1}{2}\int_0^1
\frac{\gamma_1(\rho)}{\sqrt{\epsilon_1(\rho)}}d\rho]+O(\frac{1}{k^2}),
\end{equation}
as $k\rightarrow\pm\infty$ on $0i+\mathbb{R}$.
Following this, we will show $D(k)$ is in Cartwright's class of functions. Such a
function has many advantage. We refer to Levin
\cite{Levin,Levin2} for an introduction.

\par
Alternatively, we consider the other theoretical framework \cite{Aktosun}. Because~(\ref{1.1}) is spherically symmetric,  we define
the transformation:
\begin{eqnarray}\label{19}
&&\Phi(r;k):=rw(r);\vspace{5pt}\\
&&\Phi_0(r;k):=rv(r).\label{20}
\end{eqnarray}
Therefore,~(\ref{1.1}) is transformed into the following system.
\begin{eqnarray} \label{1.9}
             \left\{%
\begin{array}{ll}
    \Phi''+k^2(\epsilon_1(r)+i\frac{\gamma_1(r)}{k})\Phi=0, & 0\leq r\leq1; \vspace{5pt} \\
    \Phi_0''+k^2(\epsilon_0+i\frac{\gamma_0}{k})\Phi_0=0, & 0\leq r\leq1; \vspace{5pt} \\
    \Phi(0)=\Phi_0(0)=0,\,\Phi(1)=\Phi_0(1),\,\Phi'(1)=\Phi_0'(1).\\
\end{array}%
\right.
\end{eqnarray}
By straightforward computation under~(\ref{19}) and~(\ref{20}),~(\ref{1.9}) is equivalent to the follwing system.
\begin{eqnarray}\label{1.8}
\left\{%
\begin{array}{ll}
    \Phi''+k^2(\epsilon_1(r)+i\frac{\gamma_1(r)}{k})\Phi=0, & 0\leq r\leq1;\vspace{5pt} \\
    \Phi(0)=0,\,\Phi_0'(1)\Phi(1)-\Phi_0(1)\Phi'(1)=0.\\
\end{array}%
\right.
\end{eqnarray}
The boundary condition of~(\ref{1.9}) or~(\ref{1.8}) implies each other.
The determinant in~(\ref{1.8}) is equivalent to~(\ref{1.7}) as well.
The zeros of the functional determinant $D(k)$ are then the
eigenvalues of~(\ref{1.8}). In this paper, we assume that two possibly different refraction indices have the same $\Phi_0(r;k)$.
In general, we note that 
\begin{eqnarray}
&&c_1=\frac{\det\left(
            \begin{array}{cc}
              y(1) & \frac{e^{ikr}}{r}|_{r=1} \\
              \{\frac{y(r)}{r}\}'|_{r=1} & \{\frac{e^{ikr}}{r}\}'|_{r=1} \\
            \end{array}
          \right)}{D(k)};\label{112}\\
&&c_2=\frac{\det\left(
            \begin{array}{cc}
              \frac{e^{ikr}}{r}|_{r=1} & -j_0(k\tilde{n}_0)\\
              \{\frac{e^{ikr}}{r}\}'|_{r=1} &  -j_0'(k\tilde{n}_0r)|_{r=1} \\
            \end{array}
          \right)}{D(k)}.\label{113}
\end{eqnarray}
 
\par To understand the analytic behavior
of the determinant $D(k)$, we  study the asymptotic solution
of~(\ref{15}). In this case, we use the theory
provided in Erdelyi \cite[p. 84]{Erdelyi}. In particular, we have a set
of fundamental solutions $y_1(r)$, $y_2(r)$ such that in a
sectorial region $S$
\begin{eqnarray}
&&y_j(r;k)=Y_j(r)[1+O(\frac{1}{k})];\label{110}\\
&&y_j'(r;k)=Y_j'(r)[1+O(\frac{1}{k})],\label{111}
\end{eqnarray}
as $|k|\rightarrow\infty$ in $S$, uniformly for $0\leq r\leq 1$
and for $\arg k$, where
\begin{equation}
Y_j(r)=\exp\{\beta_{0j}k+\beta_{1j}\},
\end{equation}
where $\beta_{0j}$, $\beta_{1j}$ satisfy
\begin{eqnarray}
&&(\beta_{0j}')^2+\epsilon_1(r)=0;\\
&&2\beta_{0j}'\beta_{1j}'+i\gamma_1+\beta''_{0j}=0.
\end{eqnarray}
\begin{eqnarray}
&&\beta_{0j}(r)=\pm
i\int_0^r\sqrt{\epsilon_1(\rho)}d\rho+E;\\
&&\beta_{1j}(r)=\mp\frac{1}{2}\int_0^r\frac{\gamma_1(\rho)}{\sqrt{\epsilon_1(\rho)}}d\rho
+\ln[\epsilon_1(r)]^\frac{-1}{4}+F,
\end{eqnarray}
where $E$, $F$ are constants. The sectorial region
$S\subset\mathbb{C}$ is characterized by the condition
\begin{equation}
\Re\{ki(\epsilon_1(r))^{\frac{1}{2}}\}\neq0.
\end{equation}
That is
\begin{equation}
S=\{k\in\mathbb{C}|\,\Im k\neq0\}.
\end{equation}
Therefore, any solution to~(\ref{15}) is of the form
\begin{equation}
y(r;k)=\alpha Y_1(r)[1+O(\frac{1}{k})]+\beta
Y_2(r)[1+O(\frac{1}{k})].
\end{equation}
We use the initial condition in~(\ref{15}) to obtain
\begin{eqnarray}\nonumber
y(r;k)&=&\frac{1}{2ik[\epsilon_1(0)\epsilon_1(r)]^{\frac{1}{4}}}
\exp\{ik\int_0^r\sqrt{\epsilon_1(\rho)}d\rho-\frac{1}{2}\int_0^r
\frac{\gamma_1(\rho)}{\sqrt{\epsilon_1(\rho)}}d\rho\}[1+O(\frac{1}{k})]\\&&
- \frac{1}{2ik[\epsilon_1(0)\epsilon_1(r)]^{\frac{1}{4}}}
\exp\{-ik\int_0^r\sqrt{\epsilon_1(\rho)}d\rho+\frac{1}{2}\int_0^r
\frac{\gamma_1(\rho)}{\sqrt{\epsilon_1(\rho)}}d\rho\}[1+O(\frac{1}{k})],
\label{1.18}
\end{eqnarray}
when $|k|\rightarrow\infty$ in $S$. Similarly, we use~(\ref{111})
to obtain the asymptotics
\begin{eqnarray}\nonumber
y'(r;k)&=&\frac{1}{2}[\frac{\epsilon_1(r)}{\epsilon_1(0)}]^\frac{1}{4}
\exp\{ik\int_0^r\sqrt{\epsilon_1(\rho)}d\rho-\frac{1}{2}\int_0^r
\frac{\gamma_1(\rho)}{\sqrt{\epsilon_1(\rho)}}d\rho\}[1+O(\frac{1}{k})]\\&&
+\frac{1}{2}[\frac{\epsilon_1(r)}{\epsilon_1(0)}]^\frac{1}{4}
\exp\{-ik\int_0^r\sqrt{\epsilon_1(\rho)}d\rho+\frac{1}{2}\int_0^r
\frac{\gamma_1(\rho)}{\sqrt{\epsilon_1(\rho)}}d\rho\}[1+O(\frac{1}{k})],
\label{1.19}
\end{eqnarray}when $|k|\rightarrow\infty$ in $S$.

Let us set
\begin{equation}\label{1.11}
A:=\sqrt{\epsilon_0},\, B:=\int_0^1\sqrt{\epsilon_1(\rho)}d\rho,\,
C:=\frac{1}{2}\frac{\gamma_0}{\sqrt{\epsilon_0}},\,
D:=\frac{1}{2}\int_0^1\frac{\gamma_1(\rho)}{\sqrt{\epsilon_1(\rho)}}d\rho.
\end{equation}
\par
When the refraction is purely real, $\gamma_1\equiv0$, the advantage is to consider the Liouville transformation of
$y(r)=y(r;k)$:
\begin{eqnarray}\label{L}
&z(\xi):=[n(r)]^{\frac{1}{4}}y(r),\mbox{ where
}\xi:=\int_0^r[n(\rho)]^{\frac{1}{2}}d\rho.
\end{eqnarray}
In particular, we define
\begin{equation}
B=  \int_0^1[n(\rho)]^{\frac{1}{2}}d\rho
\end{equation}
In this case,~(\ref{L}) becomes
\begin{eqnarray}
\left\{%
\begin{array}{ll} \label{17}
    z''+[k^2-p(\xi)]z=0,\,0<\xi<B;\vspace{5pt} \\
    z(0)=0;\,z'(0)=[n(0)]^{-\frac{1}{4}}. \\
\end{array}%
\right.
\end{eqnarray}
where
\begin{equation}
p(\xi):=\frac{n''(r)}{4[n(r)]^2}-\frac{5}{16}\frac{[n'(r)]^2}{[n(r)]^3}.
\end{equation}
From P\"{o}schel and Trubowitz \cite[p.16]{Po}, we review the
following asymptotics.
\begin{equation}\label{P1}
z(\xi;k)=\frac{\sin k\xi}{k}-\frac{\cos
k\xi}{2k^2}Q(\xi)+\frac{\sin
k\xi}{4k^3}[p(\xi)+p(0)-\frac{1}{2}Q^2(\xi)]+O(\frac{\exp[|\Im
k|\xi]}{k^4}),
\end{equation}
where $Q(\xi)=\int_0^\xi p(\xi)d\xi$. Similarly,
\begin{equation}\label{P2}
z'(\xi;k)= \cos k\xi+\frac{\sin k\xi}{2k}Q(\xi)+\frac{\cos
k\xi}{4k^2}[p(\xi)-p(0)-\frac{1}{2}Q^2(\xi)]+O(\frac{\exp[|\Im
k|\xi]}{k^3}).
\end{equation}
Before applying such asymptotics, we add a multiple
$[n(0)]^{\frac{1}{4}}$ to the solutions. We have to normalize the boundary condition in~(\ref{17}) for the solution.

We state the main result of this paper.
\begin{theorem}\label{11}
Let
\begin{eqnarray}
&&\Lambda_1:=\{z\in\mathbb{C}|\,|\arg(z)|<\epsilon\};\\
&&\Lambda_2:=\{z\in\mathbb{C}|\,|\arg(z)-\pi|<\epsilon\},\,\forall\epsilon>0.
\end{eqnarray}
Let $n_1^j$, $j=1,2$, be two unknown refraction indices and $D^j(k)$ be
the determinant corresponding to $n_1^j$ and $n_1^1(0)=n_1^2(0)$. If the zeros of $D^1(k)$
and $D^2(k)$ coincide in either $\Lambda_1$ or $\Lambda_2$, then
$\epsilon^1_1(r)\equiv \epsilon^2_1(r)$.
\end{theorem}
We use the vocabulary from entire function to describe the
distribution of the zeros of the functional determinant $D(k)$. We
refer such a theory to Levin \cite{Levin,Levin2}.
\begin{definition}
Let $f(z)$ be an entire function of order $\rho$. We use
$N(f,\alpha,\beta,r)$ to denote the number of the zeros of $f(x)$
inside the angle $[\alpha,\beta]$ and $|z|\leq r$; we define the
density function
\begin{equation}
\Delta_f(\alpha,\beta):=\lim_{r\rightarrow\infty}\frac{N(f,\alpha,\beta,r)}{r^{\rho}},
\end{equation}
and
\begin{equation}
\Delta_f(\beta):=\Delta_f(\alpha_0,\beta),
\end{equation}
with some fixed $\alpha_0\notin E$ such that $E$ is at most a
countable set.
\end{definition}
\begin{theorem}\label{13}
The determinant $D(k)$ is an entire function of order $1$ and of
type $A+B$. In particular,
\begin{equation}
\Delta_D(-\epsilon,\epsilon)=\Delta_D(\pi-\epsilon,\pi+\epsilon)=\frac{A+B}{\pi}.
\end{equation}
\end{theorem}

\section{Lemmas}
We need a few lemmas.
\begin{lemma} \label{21}
There exists a constant $M$ and $k_0>0$ such that
\begin{eqnarray}
|\frac{y(1;k)}{y'(1;k)}|<M;\,|\frac{j_0(k\tilde{n}_0)}{\partial_rj_0(k\tilde{n}_0r)|_{r=1}}|<M,\,|k|>k_0,\,k\in 0+i\mathbb{R}.
\end{eqnarray}
\end{lemma}
\begin{proof}
We start with~(\ref{1.18}). We compute the following quantity
from~(\ref{110}),~(\ref{111}) and~(\ref{1.18}).
\begin{eqnarray}
\frac{y(1;k)}{y'(1;k)}=\frac{e^{ikB-D}[1+O(\frac{1}{k})]-
e^{-ikB+D}[1+O(\frac{1}{k})]}
{e^{ikB-D}[ik\epsilon_0-\frac{\gamma_0}{2\sqrt{\epsilon_0}}][1+O(\frac{1}{k})]+
e^{-ikB+D}[ik\epsilon_0-\frac{\gamma_0}{2\sqrt{\epsilon_0}}][1+O(\frac{1}{k})]}.
\end{eqnarray}
Let $k=i\xi\in0+i\mathbb{R}^+$. Hence,
\begin{eqnarray}\nonumber
\frac{y(1;i\xi)}{y'(1;i\xi)}&=&\frac{e^{-\xi B-D}[1+O(\frac{1}{i\xi})]-
e^{\xi B+D}[1+O(\frac{1}{i\xi})]}{e^{-\xi B-D}[-\xi \epsilon_0-\frac{\gamma_0}{2\sqrt{\epsilon_0}}]
[1+O(\frac{1}{i\xi})]+
e^{\xi B+D}[-\xi\epsilon_0-\frac{\gamma_0}{2\sqrt{\epsilon_0}}][1+O(\frac{1}{i\xi})]}.
\end{eqnarray}
Therefore,
\begin{eqnarray}
|\frac{y(1;i\xi)}{y'(1;i\xi)}|&\leq&\left\{
        \begin{array}{ll}
          Ce^{-(\xi B+D)} |\frac{[1+O(\frac{1}{i\xi})]-e^{2(\xi B+D)}[1+O(\frac{1}{i\xi})]}{[1+O(\frac{1}{i\xi})]+e^{2(\xi B+D)}[1+O(\frac{1}{i\xi})]}|,&\mbox{ if }\xi>0 ;\vspace{12pt}\\
          Ce^{\xi B+D} |\frac{e^{-2(\xi B+D)}[1+O(\frac{1}{i\xi})]-[1+O(\frac{1}{i\xi})]}{e^{-2(\xi B+D)}[1+O(\frac{1}{i\xi})]+[1+O(\frac{1}{i\xi})]}|,&\mbox{ if }\xi<0,
        \end{array}
      \right.\label{2.3}
\end{eqnarray}
where $C$ is a some constant. Let us consider for $\xi>0$,
\begin{eqnarray}\nonumber
&&\lim_{\xi\rightarrow\infty}|\frac{[1+O(\frac{1}{i\xi})]-e^{2(\xi B+D)}[1+O(\frac{1}{i\xi})]}{[1+O(\frac{1}{i\xi})]+e^{2(\xi B+D)}[1+O(\frac{1}{i\xi})]}|\\
&\leq&\lim_{\xi\rightarrow\infty}\frac{|1+O(\frac{1}{i\xi})|+e^{2(\xi B+D)}|1+O(\frac{1}{i\xi})|}{||1+O(\frac{1}{i\xi})|-e^{2(\xi B+D)}|1+O(\frac{1}{i\xi})||}\\
&=&\frac{\lim_{\xi\rightarrow\infty}e^{2(\xi B+D)}\lim_{\xi\rightarrow\infty}|1+O(\frac{1}{i\xi})|+\lim_{\xi\rightarrow\infty}|1+O(\frac{1}{i\xi})|}
{\lim_{\xi\rightarrow\infty}e^{2(\xi B+D)}\lim_{\xi\rightarrow\infty}|1+O(\frac{1}{i\xi})|-\lim_{\xi\rightarrow\infty}|1+O(\frac{1}{i\xi})|}\\
&=&\frac{\lim_{\xi\rightarrow\infty}e^{2(\xi B+D)}+1}
{\lim_{\xi\rightarrow\infty}e^{2(\xi B+D)}-1}\\
&=&\lim_{\xi\rightarrow\infty}\frac{e^{2(\xi B+D)}+1}{e^{2(\xi B+D)}-1}\\
&=&1.
\end{eqnarray}
A similar analysis holds for $\xi<0$. Hence in~(\ref{2.3}), $|\frac{y(1;i\xi)}{y'(1;i\xi)}|$ vanishes along the imaginary axis. Hence, the first statement is
proved. The second one can be proved similarly.
\end{proof}
\begin{definition}
Let $f(z)$ be an integral function of finite order $\rho$ in the
angle $[\theta_1,\theta_2]$. We call the following quantity as the
indicator of the function $f(z)$.
\begin{equation} \label{2.4}
h_f(\theta):=\lim_{r\rightarrow\infty}\frac{\ln|f(re^{i\theta})|}{r^{\rho}},
\,\theta_1\leq\theta\leq\theta_2.
\end{equation}
\end{definition}
\begin{lemma}
Let $f$, $g$ be two entire functions. Then the following two
inequalities hold.
\begin{eqnarray}
&&h_{fg}(\theta)=h_{f}(\theta)+h_g(\theta),\mbox{ if one limit exists};\label{2.15}\\
&&h_{f+g}(\theta)\leq\max_\theta\{h_f(\theta),h_g(\theta)\},\label{2.16}
\end{eqnarray}
where if the indicator of the two summands are not equal at some
$\theta_0$, then the equality holds in~(\ref{2.16}).
\end{lemma}
\begin{proof}
 We can find
these in \cite[p.51]{Levin}.
\end{proof}
\begin{definition}
The following quantity is called the width of the indicator
diagram of entire function $f$:
\begin{equation}\label{d}
d=h_f(\frac{\pi}{2})+h_f(-\frac{\pi}{2}).
\end{equation}
\end{definition}

The distribution on the zeros of an entire function is described
precisely by the following Cartwright's theorem
\cite{Cartwright,Cartwright2,Levin,Levin2}. The following
statements are from Levin \cite[Ch.5, Sec.4]{Levin}.
\begin{theorem}[Cartwright]
If an entire function of exponential type satisfies one of the
following conditions:
\begin{equation}
\mbox{ the integral
}\int_0^\infty\frac{\ln|f(x)f(-x)|}{1+x^2}dx\mbox{ exists,\,and
}h_f(0)=h_f(\pi)=0,
\end{equation}
\begin{equation}
\mbox{ the integral
}\int_{-\infty}^\infty\frac{\ln|f(x)|}{1+x^2}dx<\infty.
\end{equation}
\begin{equation}
\mbox{ the integral
}\int_{-\infty}^\infty\frac{\ln^+|f(x)|}{1+x^2}dx\mbox{ exists}.
\end{equation}
\begin{equation}\label{211}
|f(x)|\mbox{ is bounded on the real axis}.
\end{equation}
\begin{equation}
|f(x)|\in\mathcal{L}^p(-\infty,\infty),
\end{equation}
then
\begin{enumerate}
    \item $f(z)$ is of class A and is of completely regular growth
    and its indicator diagram is an interval on the imaginary
    axis.
\item all of the zeros of the function $f(z)$, except possibly
those of a set of zero density, lie inside arbitrarily small
angles $|\arg z|<\epsilon$ and $|\arg z-\pi|<\epsilon$, where the
density
\begin{equation}\label{2.13}
\Delta_f(-\epsilon,\epsilon)=\Delta_f(\pi-\epsilon,\pi+\epsilon)=\lim_{r\rightarrow\infty}
\frac{N(f,-\epsilon,\epsilon,r)}{r}
=\lim_{r\rightarrow\infty}\frac{N(f,\pi-\epsilon,\pi+\epsilon,r)}{r},
\end{equation}
is equal to $\frac{d}{2\pi}$, where $d$ is the width of the
indicator diagram in~(\ref{d}). Furthermore, the limit
$\delta=\lim_{r\rightarrow\infty}\delta(r)$ exists, where
\begin{equation}
\delta(r):=\sum_{\{|a_k|<r\}}\frac{1}{a_k};
\end{equation}
\item moreover,
\begin{equation}\label{2.14}
\Delta_f(\epsilon,\pi-\epsilon)=\Delta_f(\pi+\epsilon,-\epsilon)=0,
\end{equation}
\item the function $f(z)$ can be represented in the form
\begin{equation}
f(z)=cz^me^{iC
z}\lim_{r\rightarrow\infty}\prod_{\{|a_k|<r\}}(1-\frac{z}{a_k}),
\end{equation}
where $c,m,B$ are constants and $C$ is real. \item the indicator
function of $f$ is of the form
\begin{equation}
h_f(\theta)=\sigma|\sin\theta|. \label{2.17}
\end{equation}
\end{enumerate}
\end{theorem}
The last statement is found at Levin \cite[p. 126]{Levin2}. We use
these to compute the indicator function of $D(k)$.
\begin{proposition}$D(k)$ is bounded over $0i+\mathbb{R}$ and
\begin{equation}
h_D(\theta)=(A+B)|\sin\theta|,\,\theta\in[0,2\pi].
\end{equation}
\end{proposition}
\begin{proof}
As we see from~(\ref{1.7}),
\begin{eqnarray}
D(k)=y'(1,k)\partial_r j_0(k\tilde{n}_0r)|_{r=1}\{\frac{j_0(k\tilde{n}_0)}
{\partial_rj_0(k\tilde{n}_0r)|_{r=1}}
-\frac{y(1;k)}{y'(1;k)}[1+\frac{j_0(k\tilde{n}_0)}{\partial_rj_0(k\tilde{n}_0r)|_{r=1}}]\}.
\end{eqnarray}
Moreover,~(\ref{188}) suggests that $D(k)$ is bounded over $0i+\mathbb{R}$. Hence,~(\ref{211})
is satisfied.

Now we look for~(\ref{2.17}). For $k=i\xi\in0+i\mathbb{R}$ and $|\xi|>k_0$, we have that
\begin{eqnarray}
D(i\xi)=y'(1;i\xi)\partial_rj_0(i\xi\tilde{n}_0r)|_{r=1}
\{\frac{j_0(i\xi\tilde{n}_0)}{\partial_rj_0(i\xi\tilde{n}_0r)|_{r=1}}
-\frac{y(1;i\xi)}{y'(1;i\xi)}[1+\frac{j_0(i\xi\tilde{n}_0)}{\partial_rj_0(i\xi\tilde{n}_0r)|_{r=1}}]\},
\end{eqnarray}
where the items inside the bracket are bounded by Lemma \ref{21}.
Term by term, we compute
\begin{eqnarray}\nonumber
h_{y'(1;k)}(\pm\frac{\pi}{2})&=&\lim_{\xi\rightarrow\infty}\frac{\ln|y'(1;i\xi)|}{|\xi|}\\
&=&\lim_{\xi\rightarrow\infty}\frac{\ln|e^{-\xi B-D}[-\xi\epsilon_0-\frac{\gamma_0}{2\sqrt{\epsilon_0}}]
[1+O(\frac{1}{i\xi})]+
e^{\xi B+D}[-\xi\epsilon_0-\frac{\gamma_0}{2\sqrt{\epsilon_0}}][1+O(\frac{1}{i\xi})]|}{|\xi|}\nonumber\\
&=&B.
\end{eqnarray}
Similarly,
\begin{equation}
h_{\partial_rj_0(k\tilde{n}_0r)|_{r=1}}(\pm\frac{\pi}{2})=A.
\end{equation}
Now we use~(\ref{2.15}) to obtain
\begin{equation} \label{2.24}
h_D(\pm\frac{\pi}{2})=(A+B).
\end{equation}
Hence,~(\ref{2.17}) says
\begin{equation}
h_D(\theta)=(A+B)|\sin\theta|.
\end{equation}
\end{proof}
\begin{proof}[Proof of Theorem \ref{13}]
The indicator diagram of $D(k)$ has width $2(A+B)$.~(\ref{2.13})
and~(\ref{2.24}) suggests that
\begin{equation}
\Delta_D(-\epsilon,\epsilon)=\Delta_D(\pi-\epsilon,\pi+\epsilon)=\frac{A+B}{\pi}.
\end{equation}
\end{proof}
So we have this quantity of zeros in $\Lambda_1$ and $\Lambda_2$.
\section{Stability Theorem and the Proof of Theorem \ref{11}} Let $k_j$ be a common interior
transmission eigenvalue of refraction index $n^1_1(r)$ and
$n^2_1(r)$. Let $D^i(k)$ be the corresponding functional
determinant of the index $n^i_1(r)$; $y^i(r;k)$ be the solution.
From Theorem \ref{13} and the assumption of Theorem \ref{11}, we
have
\begin{equation}
B^1=B^2,   \label{2.27}
\end{equation}
where $$B^i:=\int_0^1\sqrt{\epsilon^i_1(\rho)}d\rho.$$
\par
Let $\Phi^i(r)$ and $\Phi_0(r)$ be the function defined by index $n^i$ as in~(\ref{1.9}). 
Therefore, the boundary condition in~(\ref{1.9}),~(\ref{1.2}),~(\ref{1.3}),~(\ref{112}), and~(\ref{113}) imply that
\begin{equation}   \label{2.28}
y^1(1;k_j)=y^2(1;k_j)\mbox{ and }(y^1)'(1;k_j)=(y^2)'(1;k_j),\,\forall k_j.
\end{equation}
Let
\begin{equation}
F(k):=y^1(1;k)-y^2(1;k),
\end{equation}
which is an entire function. From~(\ref{1.18}) and~(\ref{2.4}), we
see $y^i(1;k)$ is an entire function of exponential type $B^i$ and
\begin{equation}
h_{y^i}(\theta)=B^i|\sin\theta|.
\end{equation}
Here, we see that $h_{y^i}(\theta)$ is a continuous function of
$\theta$. Therefore, applying~(\ref{2.16}) and~(\ref{2.27}),
\begin{equation} \label{2.31}
h_F(\theta)\leq B^1|\sin\theta|.
\end{equation}
Hence, the indicator diagram of $F(k)$ is equal to
\begin{equation}\label{2.32}
h_F(\frac{\pi}{2})+h_F(-\frac{\pi}{2})\leq 2B^1.
\end{equation}
However, this is not possible due to the following uniqueness
theorem for the entire function of the exponential type. This is Carlson's theorem from Levin \cite[p. 190]{Levin}.
\begin{theorem}\label{Carlson}
Let $F(z)$ be holomorphic and at most of normal type with respect
to the proximate order $\rho(r)$ in the angle $\alpha\leq\arg
z\leq\alpha+\pi/\rho$ and vanish on a set $N:=\{a_k\}$ in this
angle, with angular density $\Delta_N(\psi)$. Let
\begin{equation}\label{H}
H_N(\theta):=\pi\int_{\alpha}^{\alpha+\pi/\rho}\sin|\psi-\theta|d\Delta_N(\psi),
\end{equation}
when $\rho$ is integral. Then, if $F(z)$ is not identically zero,
\begin{equation}
h_F(\alpha)+h_F(\alpha+\pi/\rho)\geq
H_N(\alpha)+H_N(\alpha+\pi/\rho).
\end{equation}
\end{theorem}
In this paper, we consider
$\rho\equiv1,\,\alpha=-\frac{\pi}{2},\,\frac{\pi}{2}$. Let $N$
here be the common zeros either in $\Lambda_1$ or in $\Lambda_2$.
From~(\ref{2.13}) and~(\ref{2.14}), we have
\begin{equation}
H_N(-\frac{\pi}{2})+H_N(\frac{\pi}{2})=2(A+B^1).\label{333}
\end{equation}
Therefore, we conclude from~(\ref{2.32}),~(\ref{333}) and Theorem
\ref{Carlson} that $F(k)\equiv0$ and
\begin{equation} \label{3.9}
y^1(1;k)\equiv y^2(1;k).
\end{equation}
In particular,
\begin{eqnarray}
B^1=B^2;\,
D^1=D^2.
\end{eqnarray}
Similarly, we repeat the argument from~(\ref{2.28})
to~(\ref{333}), we can show
\begin{equation}\label{3.10}
(y^1)'(1;k)\equiv (y^2)'(1;k).
\end{equation}
They are the same entire function.

The zeros of $y^i(1;k)$ are the eigenvalues of the equation
\begin{eqnarray}\label{3.12}
\left\{%
\begin{array}{ll}
    (y^i)''+k^2n^i_1(r)y^i=0, & 0\leq r\leq1; \\
    y^i(0)=0, & y^i(1;k)=0, \\
\end{array}%
\right.
\end{eqnarray}
while the zeros of $(y^i)'(1;k)$ are the eigenvalues of the
equation
\begin{eqnarray} \label{3.13}
\left\{%
\begin{array}{ll}
    (y^i)''+k^2n^i_1(r)y^i=0, & 0\leq r\leq1; \\
    y^i(0)=0, & (y^i)'(1;k)=0. \\
\end{array}%
\right.
\end{eqnarray}
Let us understand more on the structure of the zero set.
\begin{definition}
Let us define
\begin{equation} \label{3.14}
\mathcal{Y}(r;k):=y^1(r;k)-y^2(r;k).
\end{equation}
\end{definition}
The following lemma holds.
\begin{lemma}\label{33}
Let $k$ be an interior transmission eigenvalue for index
$n_1^1$ and $n_1^2$. Then, $\mathcal{Y}(1;k)$ satisfies the
following boundary value problem.
\begin{eqnarray}\label{243}
&&\mathcal{Y}''(r;k)+(k^2\epsilon_1^1+ik\gamma_1^1)\mathcal{Y}(r;k)
+(k^2\epsilon_d+ik\gamma_d)y^2(r;k)=0,\,0\leq
r\leq1;\\\label{3.16}
&&(y^2)''(r;k)+(k^2\epsilon_1^2+ik\gamma_1^2)y^2(r;k)=0,\,0\leq
r\leq1;\\
&&\mathcal{Y}(1;k)=0;\label{222}\\ &&\mathcal{Y}'(1;k)=0, \mbox{
where
}\epsilon_d:=\epsilon_1^1-\epsilon_1^2,\,\gamma_d:=\gamma_1^1-\epsilon_1^2.\label{244}
\end{eqnarray}
For any $k\in\mathbb{C}$ that satisfies
\begin{equation}\label{3.18}
\left\{%
\begin{array}{ll}
    y^1(1;k)=y^2(1;k); & \\
    (y^1)'(1;k)=(y^2)'(1;k). & \\
\end{array}%
\right.
\end{equation} is an interior transmission eigenvalues of
problem~(\ref{243})-~(\ref{244}).
\end{lemma}
\begin{proof}
We note that~(\ref{222}) and~(\ref{244}) are satisfied as
in~(\ref{2.28});~(\ref{243}) and~(\ref{3.16}) are derived from~(\ref{15}).
\end{proof}

\par
The system~(\ref{243})-~(\ref{244}) has a perturbation theory to
merely finitely many eigenvalues. Such a theory is established in
\cite[Sec.2]{Cakoni}. Let us review as follows.
\par
Let us define
\begin{equation}
u:=w-v.
\end{equation}
We rewrite~(\ref{1.1}) as
\begin{eqnarray}
&&\Delta
u+(k^2\epsilon_1+ik\gamma_1)u+(k^2\epsilon_c+ik\gamma_c)v=0;\\
&&\Delta v+(k^2\epsilon_0+ik\gamma_0)v=0;\\
&&u=0,\mbox{ on }\partial B;\\
&&\frac{\partial u}{\partial\nu}=0\mbox{ on }\partial B, \mbox{
where
}\epsilon_c:=\epsilon_1-\epsilon_0;\,\gamma_c:=\gamma_1-\epsilon_0.
\end{eqnarray}
The equation makes sense for $u\in H^2_0(B)$ and $v\in L^2(B)$
such that $\Delta v\in L^2(B)$.

\par
Setting $$X(B):=H^2_0(B)\times\{v\in L^2(B)|\,\Delta v\in
L^2(B)\},$$ we can define the linear operators $\mathbb{A}$,
$\mathbb{B}_\gamma$ and $$\mathbb{D}_\epsilon:L^2(B)\times
L^2(B)\mapsto L^2(B)\times L^2(B)$$ by
\begin{eqnarray}
&&\mathbb{A}:=\left(%
\begin{array}{cc}
  \Delta_{00} & 0 \\
   0& \Delta \\
\end{array}%
\right);\\
&&\mathbb{B}_\gamma:=\left(%
\begin{array}{cc}
  i\gamma_1 & i\gamma_c \\
   0& i\gamma_0\\
\end{array}%
\right);\\
&&\mathbb{D}_\epsilon:=\left(%
\begin{array}{cc}
  \epsilon_1 & \epsilon_c \\
   0& \epsilon_0\\
\end{array}%
\right),
\end{eqnarray}
where $\Delta_{00}$ is the Laplacian acting on function in
$H^2_0(B)$, i.e. with the zero Cauchy data on $\partial B$.
Let $$p:=\left(%
\begin{array}{c}
  u \\
  v \\
\end{array}%
\right)$$ and note that the domain of $\mathbb{A}$ is $X(B)$ and
$\mathbb{A}$ is an unbounded densely defined operator in
$L^2(B)\times L^2(B)$.

Using the results in \cite{Cakoni}, we can easily show that
\begin{equation}
D_\epsilon^{-1}=\frac{1}{\epsilon_0\epsilon_1}\left(%
\begin{array}{cc}
  \epsilon_0 & -\epsilon_c \\
   0& \epsilon_1\\
\end{array}%
\right)
\end{equation}
and that the transmission eigenvalues of~(\ref{1.1}) is the
quadratic eigenvalues of the equation
\begin{equation}
\mathbb{A}p+k\mathbb{B}_\gamma p+k^2\mathbb{D}_\epsilon p=0,\,p=\left(%
\begin{array}{c}
  u \\
  v \\
\end{array}%
\right)\in L^2(B)\times L^2(B).
\end{equation}
According to \cite{Cakoni}, the eigenvalue problem~(\ref{1.1}) is
equivalent to the eigenvalue problem for the closed unbounded
operator
\begin{eqnarray}
\mathbb{T}_{\epsilon,\gamma}:=\mathbb{I}^{-1}_{\epsilon,\gamma}\mathbb{K},
\end{eqnarray}\label{}
where,
\begin{equation}
\mathbb{K}:=\left(%
\begin{array}{cc}
  \mathbb{A}& 0 \\
   0& \mathbb{I}\\
\end{array}%
\right)
\end{equation}
and
\begin{equation}
\mathbb{I}_{\epsilon,\gamma}:=\left(%
\begin{array}{cc}
  -\mathbb{B}_\gamma& -\mathbb{I} \\
   \mathbb{D}_\epsilon& 0\\
\end{array}%
\right),
\end{equation}
where $\mathbb{I}:L^2(B)\times L^2(B)\mapsto L^2(B)\times L^2(B)$
is the identity operator.  Moreover,
\begin{equation}
\mathbb{I}_{\epsilon,\gamma}^{-1}=\mathbb{D}_\epsilon^{-1}\left(%
\begin{array}{cc}
  0& -\mathbb{I} \\
   -\mathbb{D}_\epsilon& -\mathbb{B}_\gamma\\
\end{array}%
\right),
\end{equation}
which is bounded in $L^2(B)\times L^2(B)$. Therefore,
$\mathbb{T}_{\epsilon,\gamma}$ is closed in $[L^2(B)\times
L^2(B)]^2$.

\par
Now we define the operator
\begin{equation}
\mathbb{P}:=
\mathbb{T}_{\epsilon,\gamma}-\mathbb{T}_{\epsilon,\gamma=0}.
\end{equation}
Consequently,
\begin{equation}
\mathbb{P}=\left(%
\begin{array}{cc}
  0&0 \\
   0& -\mathbb{D}_\epsilon^{-1}\mathbb{B}_\gamma\\
\end{array}%
\right).
\end{equation}
There is a stability for finitely many interior transmission
eigenvalues whenever the generalized norm
\begin{equation}
\hat{\delta}(\mathbb{T}_{\epsilon,\gamma},\mathbb{T}_{\epsilon,\gamma=0})
\end{equation}
is controlled. We apply the (24) in \cite{Cakoni}.
\begin{equation}
\hat{\delta}(\mathbb{T}_{\epsilon,\gamma},\mathbb{T}_{\epsilon,\gamma=0})\leq\|\mathbb{P}\|
\leq\|\mathbb{D}_\epsilon^{-1}\mathbb{B}_\gamma\|.
\end{equation}
In \cite{Cakoni}, they perturb the non-absorbing media,
$(\gamma_1=0,\gamma_0=0)$, to conclude the existence of the
finitely many interior transmission eigenvalues in absorbing
media, in which  $(\gamma_1,\gamma_0)$ is small. In particular,
they explain as follows.
\begin{theorem}[Cakoni,
Colton, Haddar] Let $\epsilon_0\in L^\infty([0,1])$ and
satisfy $\epsilon_0(r)\geq\theta_0>0$,
$\epsilon_1(r)\geq\theta_1>0$ and
$\epsilon_c:=\epsilon_1-\epsilon_0>0$ almost every in $[0,1]$. Let
$k_j$, $j=0,\cdots,l$ be the first $l+1$ real positive
transmission eigenvalues, counted according to its multiplicity,
corresponding to~(\ref{1.1}) in non-absorbing media. Then, for any
$\sigma>0$, there exists a
\begin{equation}\label{338}
\eta:=\frac{\sup_{[0,1]}\epsilon_0+\sup_{[0,1]}\epsilon_1}
{4\inf_{[0,1]}\epsilon_0\inf_{[0,1]}\epsilon_1}\{\sup_{[0,1]}\gamma_0+\sup_{[0,1]}\gamma_1\}>0,
\end{equation}
depending on $\sigma$, such that there exists $l+1$ transmission
eigenvalues of~(\ref{1.1}) in absorbing media,
$\epsilon_1>0,\,\gamma_1>0$, in the $\sigma$-neighborhood of
$k_j$, $j=0,\cdots,l$, whenever $\eta$ is small enough.
\end{theorem}
For the application in this paper, we consider the perturbation
operator
\begin{equation}
\mathbb{P}':=
\mathbb{T}_{\epsilon,\gamma}-\mathbb{T}_{\epsilon,\gamma'}.
\end{equation}
We note that the existence of the transmission interior
eigenvalues to~(\ref{1.1}) and its distributional rule is already
described by Cartwright's theory as in Theorem \ref{13}. Hence, we
consider the perturbation from one absorbing index
$n_1=\epsilon_1+i\frac{\gamma_1}{k}$ to the other one, with
$\epsilon_1$ fixed, $n_1'=\epsilon_1+i\frac{\gamma_1'}{k}$. Now we
compute
\begin{equation}
\|\mathbb{P}'\|=\|\mathbb{T}_{\epsilon,\gamma}
-\mathbb{T}_{\epsilon,\gamma'}\|
\leq\|\mathbb{D}_\epsilon^{-1}(\mathbb{B}_{\gamma'}-\mathbb{B}_\gamma)\|.
\end{equation}
Moreover,
\begin{eqnarray}
\|\mathbb{D}_\epsilon^{-1}(\mathbb{B}_{\gamma'}-\mathbb{B}_\gamma)\|
&=&\|\frac{1}{\epsilon_0\epsilon_1}\left(%
\begin{array}{cc}
   \epsilon_0& -\epsilon_c \\
  0& \epsilon_1\\
\end{array}%
\right)[\left(%
\begin{array}{cc}
   i\gamma_1'& i\gamma_c' \\
  0& i\gamma_0 \\
\end{array}%
\right)-\left(%
\begin{array}{cc}
   i\gamma_1& i\gamma_c \\
  0& i\gamma_0 \\
\end{array}%
\right)]\|\\
&\leq&\|\frac{1}{\epsilon_0\epsilon_1}\left(%
\begin{array}{cc}
   \epsilon_0& -\epsilon_c \\
  0& \epsilon_1\\
\end{array}%
\right)\left(%
\begin{array}{cc}
   i(\gamma_1'-\gamma_1)& i(\gamma_1'-\gamma_1)\\
  0& 0\\
\end{array}%
\right)\|\\
&=&\|\frac{1}{\epsilon_0\epsilon_1}\left(%
\begin{array}{cc}
   i\epsilon_0(\gamma_1'-\gamma_1)& i\epsilon_0(\gamma_1'-\gamma_1)\\
  0& 0\\
\end{array}%
\right)\|\\
&\leq&\frac{2}{\inf_{[0,1]}\epsilon_1}\sup_{[0,1]}|\gamma_1'-\gamma_1|.
\end{eqnarray}
We proved the following stability among the interior transmission
eigenvalues.
\begin{proposition}
\label{35} Let $n_1$ be a refraction index satisfying the
assumption in~(\ref{1.1}). Let $k_j$, $j=0,\cdots,l$ be any $l+1$
interior transmission eigenvalues, counted according to its
multiplicity, corresponding to~(\ref{1.1}) in absorbing media,
$\epsilon_1>0,\,\gamma_1>0$. Then, for any $\sigma>0$, there
exists a
\begin{equation}
\eta:=\frac{1}{\inf_{[0,1]}\epsilon_1}\sup_
{[0,1]}|\gamma_1'-\gamma_1|,
\end{equation}
depending on $\sigma$, such that there exists $l+1$ interior
transmission eigenvalues of~(\ref{1.1}) in absorbing media,
$\epsilon_1>0,\,\gamma_1'>0$, in the $\sigma$-neighborhood of
$k_j$, $j=0,\cdots,l$, whenever
$\eta=\frac{1}{\inf_{[0,1]}\epsilon_1}\sup_{[0,1]}|\gamma_1'-\gamma_1|$
is small enough.
\end{proposition}
A finite collection of interior transmission eigenvalues moves continuously with respect to $\gamma_1^1$.
This explains the perturbation theory for finitely many
eigenvalues. To study the asymptotic behavior, we analyze the
asymptotic behavior of zeros of $y^1(1;k)$. We need the following
counting lemma.
\begin{lemma}\label{28}
If $|z-j\pi|\geq\delta>0,\,j\in\mathbb{Z},$ then
\begin{equation}
\exp\{|\Im z|\}<\frac{O(1)}{\delta}|\sin z|.
\end{equation}
\end{lemma}
\begin{proof}
We modify the one in \cite{Po}. Let $|z|\geq\frac{\pi}{n}$, where
$n$ is to be chosen. We discuss two cases: $0\leq
x\leq\frac{\pi}{2n}$ and $\frac{\pi}{2n}\leq x\leq\frac{\pi}{2}$.
\par
Now $x^2+y^2=|z|^2$ and $0\leq x\leq\frac{\pi}{2n}$. It implies
that
\begin{equation}
y^2=|z|^2-x^2\geq(\frac{\pi}{n})^2-(\frac{\pi}{2n})^2=\frac{3}{4}\frac{\pi^2}{n^2}.
\end{equation}
Hence,
\begin{equation} \label{2.42}
\cosh^2y\geq 1+y^2\geq
1+\frac{3}{4}\frac{\pi^2}{n^2}\geq[1+\frac{1}{4}(\frac{\pi}{n})^2]\cos^2x.
\end{equation}
For $\frac{\pi}{2n}\leq x\leq\frac{\pi}{2}$, we see
$\cos{\frac{\pi}{2n}}\geq\cos x\geq\cos\frac{\pi}{2}=0$ and
\begin{equation}
\frac{1}{\cos
x}=1+\frac{1}{2!}x^2+(\frac{1}{4}-\frac{1}{4!})x^4+\cdots\geq1+\frac{x^2}{2},\mbox{
when }x\mbox{ is small}.
\end{equation}
Hence, squaring on both sides,
\begin{equation}
\frac{1}{\cos^2x}\geq1+x^2+\frac{x^4}{4},\mbox{ when }x\mbox{ is
small}.
\end{equation}
Using this with $\cosh^2y\geq1$,
\begin{equation}\label{2.44}
\cosh^2y\geq\frac{1}{\cos^2\frac{\pi}{2n}}\cos^2x
\geq[1+(\frac{\pi}{2n})^2+\frac{1}{4}(\frac{\pi}{2n})^4]\cos^2x\geq[1+\frac{1}{4}(\frac{\pi}{n})^2]\cos^2x.
\end{equation}
To~(\ref{2.42}) and~(\ref{2.44}), we use $|\sin
z|^2=\cosh^2y-\cos^2x$. Hence,
\begin{equation}
|\sin
z|^2\geq\cosh^2y-(1+\frac{1}{4}\frac{\pi^2}{n^2})^{-1}\cosh^2y\geq(1-(1+\frac{1}{4}\frac{\pi^2}{n^2})^{-1})
\frac{e^{2|y|}}{4}.
\end{equation}
Hence, letting
$C_n=2[1-(1+\frac{1}{4}\frac{\pi^2}{n^2})^{-1}]^{-\frac{1}{2}}=\frac{2n\sqrt{\pi^2+4}}{\pi}=O(n)$,
we conclude
\begin{equation}
\exp\{|\Im y|\}<C_n|\sin z|. \label{3.35}
\end{equation}
Considering $\delta=\frac{1}{n}$, this proves the lemma.
\end{proof}
\begin{proposition}\label{37}
Let $z_j$, be the zeros of $y(1;k)$ and $z'_j$ be the zeros of
$y'(1;k)$.  The following asymptotics hold.
\begin{equation}\label{3.51}
z_j=\frac{j\pi}{B}-\frac{iD}{B}+O(\frac{1}{j}),\,j\in\mathbb{Z};
\end{equation}
\begin{equation} \label{3.52}
z_j'=\frac{(j-\frac{1}{2})\pi}{B}-\frac{iD}{B}+O(\frac{1}{j}),\,j\in\mathbb{Z}.
\end{equation}
\end{proposition}
\begin{proof}
We consider the zeros of
$k[\epsilon_1(0)\epsilon_1(1)]^{\frac{1}{4}}y(1;k)$ instead. We
observe from~(\ref{1.18}) that
\begin{eqnarray}
|k[\epsilon_1(0)\epsilon_1(1)]^{\frac{1}{4}}y(1;k)-\sin\{kB+iD\}|&=&O(\frac{1}{|k|})O(\exp\{|\Im k|B\})\\
&=&O(\frac{1}{|k|})\exp\{|\Im k|B\},\,\Im k\neq0. \label{3.56}
\end{eqnarray}
Firstly we apply the Rouch\'{e}'s theorem inside the strip with
boundary: $\Re k=\frac{(j-\frac{1}{2})\pi}{B}$, $\Re
k=\frac{(j+\frac{1}{2})\pi}{B}$. For this purpose, we choose $M$
large in the
\begin{equation}
C_{j/M}=O(j/M);\,|k|=j
\end{equation}
such that~(\ref{3.35}) and~(\ref{3.56}) imply
\begin{eqnarray}\label{360}
|k[\epsilon_1(0)\epsilon_1(1)]^{\frac{1}{4}}y^1(1;k)-\sin\{kB+iD\}|<|\sin\{kB+iD\}|,\,\Im
k\neq0.
\end{eqnarray}
The contour applies even without the behavior at $\Im k=0$. Hence,
we know $y^1(1;k)$ has a zero inside the strip.
\par
Secondly, we apply Rouch\'{e}'s theorem again over a contour as
the boundary of the vertical strip with a punctured hole: $\Re
k=\frac{(j-\frac{1}{2})\pi}{B}$, $\Re
k=\frac{(j+\frac{1}{2})\pi}{B}$ and
$|k-(\frac{j\pi}{B}-i\frac{D}{B})|=\rho$. We may choose $j$ large
and some $M$ such that
\begin{equation}\label{361}
\rho>\frac{M}{jB}\mbox{ and }|C_{j/M}O(\frac{1}{j})|<1.
\end{equation}
In such a punctured strip,
\begin{eqnarray}
|k[\epsilon_1(0)\epsilon_1(1)]^{\frac{1}{4}}y^1(1;k)-\sin\{kB+iD\}|<|\sin\{kB+iD\}|,\,\Im
k\neq0,
\end{eqnarray}
by Lemma \ref{28}. Hence, the zeros of
$k[\epsilon_1(0)\epsilon_1(1)]^{\frac{1}{4}}y^1(1;k)$ are the same
ones of the $\sin\{kB+iD\}$ inside the strip, but outside the
$\rho-$ball centered at $\frac{j\pi}{B}-i\frac{D}{B}$, for $j$
large. There is no zero there.
\par
Hence, for each large $j$, we know there is a zero in each strip
and inside the $\rho-$ball centered at
$\frac{j\pi}{B}-i\frac{D}{B}$. That is
\begin{equation}
z_j=\frac{j\pi}{B}-\frac{iD}{B}+O(\frac{1}{j}),\,j\in\mathbb{Z}.
\end{equation}
\par
Similarly, from~(\ref{1.19}) we have
\begin{eqnarray}\label{3.60}
|k[\frac{\epsilon_1(0)}{\epsilon_1(1)}]^{\frac{1}{4}}y'(1;k)-\cos\{kB+iD\}|
=O(\frac{1}{|k|})\exp\{|\Im k|B\},\,\Im k\neq0.
\end{eqnarray}
We apply Lemma \ref{28} as follows:
$|z+\frac{\pi}{2}-j\pi|\geq\delta>0,\,j\in\mathbb{Z},$ then
\begin{equation} \label{3.61}
\exp\{|\Im z|\}<\frac{O(1)}{\delta}|\cos z|.
\end{equation}
Apply Rouch\'{e}'s theorem again with~(\ref{3.60})
and~(\ref{3.61}), we prove the asymptotics~(\ref{3.52}).
\end{proof}

\par 
The following lemma is classic \cite[Ch.\,2]{Po}. 
\begin{lemma}\label{38}
When $n_1(r)=\epsilon_1(r)$, then all of zeros of $y(1;z)$, $y'(1;z)$ are real.
\end{lemma}
\begin{proof}
We apply the theorem in \cite[Theorem 1]{Duffin}.
Since
\begin{eqnarray}\label{3.65}
y(1;z)=\frac{1}{2iz[\epsilon_1(0)\epsilon_0]^{\frac{1}{4}}}
\exp\{izB\}[1+O(\frac{1}{z})]
- \frac{1}{2iz[\epsilon_1(0)\epsilon_0]^{\frac{1}{4}}}
\exp\{-izB\}[1+O(\frac{1}{z})].
\end{eqnarray}
Therefore, we rewrite this as
\begin{equation}\label{3.67}
z[\epsilon_1(0)\epsilon_0]^{\frac{1}{4}}y(1;z):=\sin{zB}-f(z),
\end{equation}
where
\begin{equation}
f(z)=-\frac{1}{2i}e^{izB}O(\frac{1}{z})
+\frac{1}{2i}e^{-izB}O(\frac{1}{z}).
\end{equation}
It is easy to deduce that $f(z)$ is an exponential function of at most type $B$. By the uniqueness of~(\ref{15}), $y(1;z)$ is real on the real axis. Hence,~(\ref{3.67}) implies that $f(z)$ is real on the real axis. Without loss of generality, we assume $|f(z)|\leq1$ on the real. If not, we scale~(\ref{3.67}) with respect to $z$ as follows. In particular, on the real axis,
\begin{eqnarray}
|f(z M)|\leq\frac{1}{2}O(\frac{1}{ z M})+\frac{1}{2}O(\frac{1}{ z M})
=O(\frac{1}{M z})=\frac{1}{M}O(\frac{1}{ z}).
\end{eqnarray}
We choose $M>1$ large such that
\begin{eqnarray}
|f(z M)|\leq1,\,\mbox{ on the real axis}.
\end{eqnarray}
Therefore, \cite[Theorem 1]{Duffin} implies that
\begin{equation}
z M[\epsilon_1(0)\epsilon_0]^{\frac{1}{4}}y(1; z  M )=\sin{ z M B}-f( z M )
\end{equation}
has only real zeros or vanishes identically. Thus, $y(1;z)$ has only real zero or vanishes identically. The proof for $y'(1;z)$ is similar.
\end{proof}

Now we return to the result~(\ref{3.9}) and~(\ref{3.10}). We set the absorbing part of index $n_1^{i}:=\epsilon_1^i+i\frac{\gamma_1^i}{z}$, $i=1,2$, to be zero. Let $y^{i,0}(1;z)$ be the solution to such an index. They are represented by the construction of~(\ref{1.18}) and~(\ref{1.19}) or~(\ref{P1}) and~(\ref{P2}). However, we can merely deduce that
\begin{eqnarray}
y^{i,0}(1;z)=\frac{1}{2iz[\epsilon_1^i(0)\epsilon_0]^{\frac{1}{4}}}
\exp\{izB^{i}\}[1+O(\frac{1}{z})]
- \frac{1}{2iz[\epsilon_1^i(0)\epsilon_0]^{\frac{1}{4}}}
\exp\{-izB^{i}\}[1+O(\frac{1}{z})],
\end{eqnarray}
where $B^1=B^2$. We need to prove $y^{1,0}(1;z)\equiv y^{2,0}(1;z)$.

Let us define
\begin{equation}\label{3.71}
Q(z):=\frac{y^{1,0}(1;z)}{y^{2,0}(1;z)}.
\end{equation}
Firstly, we apply Lemma \ref{38}: $Q(z)$ is holomorphic outside $0i+\mathbb{R}$.
We use~(\ref{3.61}) and~(\ref{360}) to find a sequence of real numbers $\rho_j$ such that $\rho_j=O(\frac{1}{j})$ and that each zero of
$y^{1,0}(1;z)$ and $y^{2,0}(1;z)$ is located inside the ball $\Omega_j:=\{z|\,|z-\frac{j\pi}{B^1}|\leq\rho_j\}$ for all $j$. We substitute the asymptotics~(\ref{P1}) into $Q(z)$. We obtain
\begin{eqnarray}\nonumber
Q(z)&=&\frac{\frac{\sin B^1z}{z}+O(\frac{e^{|\Im z|B^1}}{z^2})}{\frac{\sin B^2z}{z}+O(\frac{e^{|\Im z|B^2}}{z^2})}\\&=&\frac{\sin B^1z[1+O(\frac{1}{z})]}{\sin B^2z[1+O(\frac{1}{z})]},\,|z-\frac{j\pi}{B^1}|=\rho_j,
\end{eqnarray}
where we have applied the~(\ref{3.35}). Since $B^1=B^2$,
\begin{eqnarray}\nonumber
Q(z)=\frac{\sin B^1z[1+O(\frac{1}{j})]}{\sin B^2z[1+O(\frac{1}{j})]}=1+O(\frac{1}{j}),\,|z-\frac{j\pi}{B^1}|=\rho_j.
\end{eqnarray}
Thus,
\begin{eqnarray}\nonumber
\sup_{|z-\frac{j\pi}{B}|=\rho_j}|Q(z)-1|=O(\frac{1}{j}).
\end{eqnarray}
We consider the maximum principle in complex analysis inside $\Omega_j$'s for large $j$. Hence, the limit on the left vanishes as $j\rightarrow\infty$.
Therefore,
\begin{equation}
Q(z)\rightarrow1, \mbox{ as }z\rightarrow0i\pm\infty\mbox{ on }0i+\mathbb{R}.
\end{equation}
Thus, there are only finitely many zeros or poles of $Q(z)$ on $0i+\mathbb{R}$. Let $\{\sigma_j\}_{j=1}^{N_1}$, $\{\xi_k\}_{k=1}^{N_2}$ be the zeros and poles of $Q(z)$ on $0i+\mathbb{R}$. They must be zeros of $y^{1,0}(1;k)$, $y^{2,0}(1;k)$. Rouche's theorem and~(\ref{360}) imply that $N_1=N_2$. Henceforth, $y^{1,0}(1;k)$ and $y^{2,0}(1;k)$ have all common zeros after their $N_1$-th zero. Let us denote them as
\begin{equation}
\{\tau_1,\tau_2,\ldots,\}.
\end{equation}

Let us define
\begin{equation}\label{3.73}
q(z):=Q(z)\frac{\prod_k(z-\xi_k)}{\prod_k(z-\sigma_k)}.
\end{equation}
Hence, $q(z)$ is entire and
\begin{equation}\label{3.72}
|q(z)|\rightarrow1, \mbox{ as }z\rightarrow0i\pm\infty.
\end{equation}

\par
Moreover, we apply~(\ref{2.15}) to find
\begin{equation}\label{3.70}
h_Q(\theta)=h_{y^{1,0}}(\theta)-h_{y^{2,0}}(\theta)=B^1\sin\theta-B^2\sin\theta=0;
\end{equation}
\begin{equation}\label{3.75}
h_Q(\theta)=h_q(\theta).
\end{equation}
We consider the~(\ref{3.72}),~(\ref{3.70}) and~(\ref{3.75}) under Phragm\'{e}n-Lindel\"{o}f's theorem \cite[p.\,38]{Levin} and \cite{Titch} to deduce that $q(z)$ is uniformly bounded in $\mathbb{C}^+$. A similar argument works in
$\mathbb{C}^-$. Therefore, Liouville's theorem implies that $q(z)$ is a constant that is $q(z)\equiv1$. From~(\ref{3.73}),
\begin{equation}\label{3.79}
 \prod_{k=1}^{N_1}(z-\xi_k)y^{1,0}(1;z)\equiv \prod_{k=1}^{N_1}(z-\sigma_k)y^{2,0}(1;z).
\end{equation}
\par
We repeat every equation since~(\ref{3.71}) similarly to conclude that
\begin{equation}\label{3.80}
\prod_{l=1}^{N_1'}(z-\xi_l')(y^{1,0})'(1;z)\equiv \prod_{l=1}^{N_1'}(z-\sigma_l')(y^{2,0})'(1;z),
\end{equation}
where $\{\xi_l'\}$ is the zeros of $(y^{2,0})'(1;z)$; $\{\sigma_l'\}$ of $(y^{1,0})'(1;z)$. Besides these finite exceptional points, $\{\tau_m'\}$ is the common zero set.
\par
To prove a sharper identity, we apply Proposition \ref{35} under the equation~(\ref{3.9}) and~(\ref{3.10}) by letting $\gamma_1^1\downarrow0$: There exists $\{\eta_1,\eta_2,\ldots,\eta_{M_1}\}$ such that they satisfy
\begin{eqnarray}&&
y^{1,0}(1;\eta_j)=y^{2,0}(1;\eta_j)\neq0;\\
&&(y^{1,0})'(1;\eta_j)=(y^{2,0})'(1;\eta_j)\neq0,
\end{eqnarray}
where $\eta=1,\ldots,\eta_{M_1}$ as the interior transmission eigenvalues for index $n_1^1=\epsilon_1^1$ and $n_1^2=\epsilon_1^2$. We note that the density of Dirichlet eigenvalues is $\frac{B^1}{\pi}$, but the density of transmission eigenvalues is $\frac{A+B^1}{\pi}$. We have more interior eigenvalues. Let us choose $M_1>N_1$ and $M_1>N_1'$. Therefore,
\begin{eqnarray}
&&\prod_{k=1}^{N_1}(z-\xi_k)\equiv \prod_{k=1}^{N_1}(z-\sigma_k);\\
&&\prod_{l=1}^{N_1'}(z-\xi_l')\equiv \prod_{l=1}^{N_1'}(z-\sigma_l'),
\end{eqnarray}
by the fundamental theorem of algebra. Thus,
\begin{equation}\label{3.81}
y^{1,0}(1;z)\equiv y^{2,0}(1;z);\,(y^{1,0})'(1;z)\equiv (y^{2,0})'(1;z).
\end{equation}

\par
As~(\ref{3.81}) has shown, we have reduced an absorbing problem to a non-absorbing one.
The rest of proof is the application of the inverse rod density
problem  which we refer to the Corollary 2.9 in \cite{Aktosun}. This proves the Theorem
\ref{11}.

\end{document}